\documentclass[a4paper,12pt,leqno]{amsart} 
\usepackage{amsmath,amsthm,amssymb}  
\usepackage{mathdots}
\usepackage{graphicx}

\numberwithin{equation}{section}

\theoremstyle{plain}
\newtheorem{theorem}{Theorem}[section]
\newtheorem{proposition}[theorem]{Proposition}         
\newtheorem{corollary}[theorem]{Corollary} 
\newtheorem{lemma}[theorem]{Lemma} 
\newtheorem{definition}[theorem]{Definition}  

\theoremstyle{definition}  
\newtheorem{example}[theorem]{Example} 
\newtheorem{remark}[theorem]{Remark} 

\newcommand{\C}{\mathbb C}   
\newcommand{\R}{\mathbb R}
\newcommand{\Z}{\mathbb Z}
\newcommand{\N}{\mathbb N}  
\newcommand{\Q}{\mathbb Q}  

\newcommand{\al}{\alpha}

\newcommand{\de}{\delta}
\newcommand{\la}{\lambda}
\newcommand{\si}{\sigma} 
\newcommand{\La}{\Lambda}
\newcommand{\eps}{\epsilon}

\renewcommand{\th}{\theta}

\newcommand{\om}{\omega}

\DeclareMathOperator{\tr}{tr}

\DeclareMathOperator{\diag}{diag}

\newcommand{\SL}{\textrm{SL}}

\renewcommand{\sl}{\frak s\frak l}

\newcommand{\SU}{\textrm{SU}}

\newcommand{\su}{\frak s\frak u}

\newcommand{\g}{{\frak g}}
\newcommand{\h}{{\frak h}}

\newcommand{\fA}{ \mathfrak{A} }

\newcommand{\no}{\noindent}

\newcommand{\st}{\ \vert\ }   
\renewcommand{\ll}{\lq\lq}
\newcommand{\rr}{\rq\rq\ }
\newcommand{\rrr}{\rq\rq}

\newcommand{\na}{\nabla}
\newcommand{\bp}{\begin{pmatrix}} 
\newcommand{\ep}{\end{pmatrix}} 
\newcommand{\bsp}{\left(\begin{smallmatrix}} 
\newcommand{\esp}{\end{smallmatrix}\right)}

\newcommand{\zzb}{ {z\bar z}  }
\newcommand{\ttb}{ {t\bar t}  }
\renewcommand{\i}{ {\scriptscriptstyle\sqrt{-1}}\, }
\newcommand{\ii}{ {\scriptstyle\sqrt{-1}}\, }

\newcommand{\Mzw}{ M^{(0)} }
\newcommand{\Mwd}{ M^{(0)}_{\diag}  }

\newcommand{\pp}{ p^\prime  }

\begin{document}     

\title[Representations and Stokes matrices]{Positive energy representations of affine algebras 
and 
Stokes matrices of the affine Toda equations
}  
   
\author{Martin A. Guest and Takashi Otofuji}    

\date{}   

\maketitle 

\begin{abstract}We give a construction which produces a
positive energy representation of the affine Lie algebra 
$\widehat{\sl}_{n+1}\C$ 
from the Stokes data of a solution of the tt*-Toda equations.  The
construction appears to play a role in conformal field theory.
We illustrate this with
several examples: the fusion ring, $W$-algebra minimal models 
(Argyres-Douglas theory),
as well as topological-antitopological fusion itself.
\end{abstract}

\section{Introduction}\label{intro}

In this article we give a purely mathematical construction which relates 

--- integrable p.d.e.\ (affine Toda equations), 

--- Stokes data of linear meromorphic o.d.e., and 

--- representations of
 infinite-dimensional Lie algebras.  
 
 \no By \ll purely mathematical\rr we mean that the construction {\em a priori} does not depend on concepts from physics. Nevertheless, our project was indeed motivated by physical ideas, originating from topological-antitopological fusion and quantum cohomology.  It seems to have a role --- as a
rather special example, at least --- in some mathematical aspects of conformal field theory. 

Our principal motivation is the tt*-Toda equations (tt* equations of Toda type). We
review these equations and their (global) solutions in section \ref{tt}. In
section \ref{sto} we give a Lie-theoretic description of the Stokes data of these solutions --- the main
technical ingredient here is from \cite{GH2}. Our construction of
positive energy representations of affine Lie algebras 
from Stokes data is given in section \ref{alc}.  

In section
\ref{cft} we give several applications of this construction in conformal field theory.  We expect
that this material could be expanded and developed further. 
For example, the ingredients of the construction all occur in the ODE/IM Correspondence, and we would expect a relation with that intriguing area.

\no{\em Acknowledgements:\  }  The first author was partially supported by JSPS grant 18H03668.

\section{The tt*-Toda equations}\label{tt}

We begin with a brief review of the tt*-Toda equations, in order to motivate our
main construction in section \ref{alc}.

The tt* equations were introduced by Cecotti-Vafa in their study of 
$N=2$ supersymmetric field theory (see \cite{CeVa91},\cite{CeVa93}).  They discussed
several examples, the most prominent being the tt* equations of \ll Toda type\rrr, 
or tt*-Toda equations.  These are
\begin{equation}\label{ost}
 2(w_i)_{\ttb}=-e^{2(w_{i+1}-w_{i})} + e^{2(w_{i}-w_{i-1})}, \ 
 w_i:\C^\ast\to\R, \ 
 i\in\Z
\end{equation}
where the real functions $w_0,\dots,w_n$ satisfy
$w_i=w_{i+n+1}$,
$w_i=w_i(\vert t\vert)$,
$w_i+w_{n-i}=0$.   

From physical considerations, Cecotti-Vafa predicted the existence of solutions of (\ref{ost}) with certain properties. As a
first step in this direction, 
the following statement was proved
in \cite{GuLi14},\cite{GIL1},\cite{GIL2},\cite{GIL3},\cite{MoXX},\cite{Mo14}.

\begin{theorem}\label{GIL} For each $N>0$, there is a one-to-one correspondence
between solutions of (\ref{ost}) on $\C^\ast$
and $\sl_{n+1}\C$-valued $1$-forms $\eta(z) \,dz$ on (the universal cover of) $\C^\ast$, where
\begin{equation}\label{GILhiggs}
\eta(z)=
\left(
\begin{array}{c|c|c|c}
\vphantom{\dfrac12}
 & & &  z^{k_0}
 \\
\hline
\vphantom{\dfrac12}
z^{k_1}  &  &   &  \\
\hline
\vphantom{\dfrac12}
  & \  \ddots\  &  & 
\\
\hline
\vphantom{\dfrac12}
 & &z^{k_n}  &  \!
\end{array}
\right).
\end{equation}
Here the $k_i$ are real numbers 
satisfying  $k_i\in[-1,\infty)$, $n+1+\sum_{i=0}^n k_i=N$, and
$k_i=k_{n-i+1}$ for $i=1,\dots,n$.
The variable $z$ of (\ref{GILhiggs}) is related to 
the variable $t$ of (\ref{ost}) by
$t=\tfrac{n+1}N z^{  \frac N{n+1} }$.
\end{theorem}

A more meaningful correspondence is obtained by introducing 
real numbers $m_0,\dots,m_n$ with $m_i+m_{n-i}=0$.   The $m_i$ are defined by:
\begin{equation}\label{mandk}
m_{i-1}-m_i+1=\tfrac{n+1}N(k_i+1).
\end{equation}
(We make the
convention that $m_i=m_{i+n+1}$.)
In terms of the $m_i$, we have a one-to-one  correspondence
between solutions of (\ref{ost}) and the convex polytope
\[
\{ m=\diag(m_0,\dots,m_n) \st m_{i-1}-m_i+1\ge 0, m_i+m_{n-i}=0 \}.
\]
Then the relation between $m_0,\dots,m_n$ and $w_0,\dots,w_n$ is simply given by the asymptotics of the solution at $t=0$, namely
\[
w_i\sim -m_i \log \vert t\vert \quad \text{(as $t\to 0$)}.
\] 
Writing $w=\diag(w_0,\dots,w_n)$, we have $w\sim -m \log \vert t\vert$.
With this notation (\ref{mandk}) is equivalent to
$z^{\frac N{n+1} m} \eta(z)z^{-\frac N{n+1} m}=z^{\frac 1{n+1} \sum_{i=0}^n k_i} \eta(1)$.
Thus, the $m_i$ arise simply through \ll balancing\rr the $k_i$.

It is well known that solutions of the Toda equations correspond to
certain kinds of harmonic maps.
The above relation between solutions $w$ and $1$-forms $\eta(z) \, dz$ is, in fact,  an example of the
generalized Weierstrass representation (or DPW representation) for harmonic maps of surfaces into symmetric spaces \cite{DoPeWu98}.  This is based on the loop group Iwasawa factorization \cite{PrSe86}. 

We review this construction very briefly, referring to section 2 of \cite{GIL3} for details.
Introducing a loop parameter $\la\in S^1$, one can solve the complex o.d.e.
\[
L^{-1} L_z = \tfrac1\la \eta, \quad L\vert_{z=0} = I
\]
near $z=0$ (at least, if all $k_i>-1$, which is the case needed in this article).  Then
the Toda equations (\ref{ost}) are the zero curvature
condition for the $1$-form\footnote{
For simplicity we are now modifying the notation of \cite{GIL3}.  In \cite{GIL3}, $\al$ denotes
a gauge equivalent $1$-form $\al=(L_\R G)^{-1} d(L_\R G)$, which 
has the same zero curvature condition. 
In \cite{GIL3}, the
nonzero entries of $\eta$ are $c_i z^{k_i}$ rather than $z^{k_i}$, and
the global solutions are
given by certain specific $c_i$, but we may
set all $c_i=1$ at the expense of modifying the Iwasawa factorization.
}
$\al=L_\R^{-1} dL_\R$, where
$L=L_\R L_+$ is a suitable Iwasawa factorization.  It can be shown that
$L_+=b+O(\la)$ where $b=\diag(b_0,\dots,b_n)$ and all $b_i>0$. 
Then one  defines $w_i= \log b_i - m_i\log \vert t\vert$.
So far this discussion is local (near $z=0$), and straightforward; the
nontrivial aspect of Theorem \ref{GIL} is that the local solutions are in fact globally defined for all
$0<\vert t\vert < \infty$. 

A summary of results related to Theorem \ref{GIL}, with some physical background, 
can be found in \cite{Gu21}.

\begin{remark} 
There are various equivalent forms of (\ref{ost}), which depend on the definition
of $w_i$ in terms of $b_i$, and whether $t$ or $z$ is used.  In terms of $z$, for
example, $w_i=\log b_i$ gives
\begin{equation*}
 2(w_i)_{\zzb}=- {\vert z^{k_{i+1}}\vert}^2 e^{2(w_{i+1}-w_{i})} +  {\vert z^{k_{i}}\vert}^2 e^{2(w_{i}-w_{i-1})}.
\end{equation*}
We use the $t$ version (\ref{ost}) for consistency with \cite{GIL3}.  
\qed
\end{remark}

\section{Stokes data}\label{sto}

The radial condition
$w=w(\vert t\vert)$ leads to another, quite different, interpretation of equation (\ref{ost}): it
is the condition that a certain meromorphic connection $\hat\al$  in the variable $\la\in\C^\ast$
has the property that its monodromy data is independent of $z$.   Details of this formulation 
can be found in section 2 of \cite{GIL3}. 

The isomonodromic connection $\hat\al$ has  (single-valued)
coefficients which are holomorphic for $\la\in\C^\ast$, 
but has poles of order $2$ at $\la=0$ and $\la=\infty$.
Its monodromy data consists of Stokes matrices relating solutions on sectors
at each pole and \ll connection matrices\rr  relating solutions at $\la=0$ and $\la=\infty$.
We refer to chapter 1 of \cite{FIKN06} for these concepts from o.d.e.\ theory. 

More generally, it was shown by Dubrovin \cite{Du93} that the tt* equations always have an  isomonodromic formulation
with the same pole structure (poles of order $2$ at $\la=0$ and $\la=\infty$). 
Monodromy data of
such connections can be hard to calculate, but,
for the tt*-Toda equations, calculations are facilitated by the close relation between
the Toda equations and Lie theory.  

This monodromy data was calculated in \cite{GIL2},\cite{GIL3},\cite{GH1},\cite{GH2}.
The result permits another characterization of the (global) solutions of (\ref{ost}), as an
alternative to the asymptotic data $m$, as in the next theorem.  Such a characterization had also been predicted by Cecotti-Vafa.  

In fact the  connection matrices turn out to be the same for all (global) solutions, so we may ignore them here.
The Stokes data may be specified efficiently as follows:

\begin{theorem}\label{GILstokes} For each $N>0$, there is a one-to-one correspondence
between solutions of (\ref{ost}) on $\C^\ast$
and
$n$-tuples of \ll Stokes parameters\rr
\[
s=(s_1,\dots,s_n)\in \R^n
\]
with 
$s_i=s_{n-i+1}$.  Explicitly, $s_i$ is
the $i$-th symmetric function of the $n+1$ numbers
$
e^{ (2m_0+n)\frac{\pi\i}{n+1}  },
e^{ (2m_1+n-2)\frac{\pi\i}{n+1}  },
\dots,
e^{ (2m_n-n)\frac{\pi\i}{n+1}  }.
$
\end{theorem}

The Stokes parameters are (up to sign) the coefficients of
the characteristic polynomial of a certain \ll $(n+1)$-th root of monodromy\rr matrix 
$\Mzw$, from which the Stokes matrices can be read off.  
The definition of
$\Mzw$, and the proof of the theorem,  is given in section 6 of \cite{GH2}.  
It
should be noted that the Stokes parameters are canonical, unlike the Stokes matrices themselves,
which depend on various choices. This is important for physical applications. 

As it will play a role in the next section, we note the following Lie-theoretic 
property of $\Mzw$, which implies the formula for $s_i$ just given.

\begin{theorem}\label{qqpi} Assume that all  $k_i>-1$. Then $\Mzw$ is conjugate to
the diagonal matrix 
\[
\Mwd=
e^{\frac{2\pi\i}{n+1}(m+\rho)}
\]
where
$\rho=\diag(\frac n2, \frac n2-1,\dots,-\frac n2)$.
\end{theorem}

\begin{proof} When all  $k_i>-1$, $\Mzw$ is diagonalizable, so the
result follows from Proposition 6.9 of \cite{GH2}.  In the
notation of \cite{GH2} and \cite{Ko59}, $\rho=x_0$.
\end{proof}

\begin{remark}\label{integerk}
In the next section we shall restrict further to $k_i\in\Z_{\ge 0}$. Then
$\eta(z) \, dz$ is a holomorphic connection form on $\C$, with a
pole at $z=\infty$.  The Stokes data turns out to be {\em equivalent} to $s=(s_1,\dots,s_n)$. However,
the pole does not have semisimple residue, and its order depends on the values of
$k_0,\dots,k_n$ so the Stokes data
is harder to extract.  These difficulties may be dealt with by classical o.d.e.\ methods, but 
it is more efficient to use homogeneity and replace the $1$-form 
$\om=\frac1\la \eta(z)\, dz$ by a meromorphic
$1$-form $\hat\om$ in the variable $\la$, just as we replaced $\al$ by $\hat\al$.  
The $1$-form $\hat\om$ always has a semisimple pole of order $2$ at $\la=0$ and a pole of order $1$ at $\la=\infty$.
The Iwasawa factorization shows that the Stokes data of $\hat\al$ and $\hat\om$ are the same at $\la=0$.
The Stokes data of $\hat\om$ is easily calculated.
We note that $\al$ defines a harmonic bundle, and $\om$ a corresponding Higgs bundle; this is
the point of view of \cite{MoXX}, \cite{Mo14}.  In our notation the harmonic
metric is $e^{-2w}$ (cf.\ section 4.2 of \cite{Gu21}).
\qed
\end{remark}

\begin{remark}\label{antisym}
As stated in Remark \ref{integerk}, the Stokes data of the 
tt*-Toda connection  is the same as the Stokes data of the 
holomorphic connection, and so the formula in Theorem \ref{qqpi}
applies also to the 
holomorphic connection.  
Although the condition $k_i=k_{n-i+1}$ was
imposed in \cite{GH2}, the proof of the formula in Theorem \ref{qqpi}
makes no use of this condition; it remains valid for arbitrary $k_i\ge -1$.  
We shall need this fact in sections \ref{alc} and \ref{cft}.
\qed
\end{remark}

\section{Lie algebra representations from Stokes data}\label{alc}

We shall now establish a relation between

--- solutions $w$ of (\ref{ost}) with all $k_i\in\Z_{\ge 0}$

--- positive energy representations of the affine Lie algebra $\widehat{\sl}_{n+1}\C$

\no by utilizing the matrices $\Mzw$ (i.e.\ the Stokes data of $w$). 

For the Lie algebra $\g=\sl_{n+1}\C$ we choose 
the standard diagonal Cartan subalgebra $\h=\{ \diag(h_0,\dots,h_n) \st h_i\in\C, \sum_{i=0}^n h_i = 0\}$, and roots 
$x_i-x_j$ $(0\le i\ne j\le n)$, where $x_i:\diag(h_0,\dots,h_n)\mapsto h_i$.
We take $\al_i=x_{i-1}-x_i$ $(1\le i\le n)$ as simple roots, then $\psi=x_0-x_n$ is the highest root. 
Here we are using the Lie-theoretic conventions\footnote{
In this section we drop the requirement
$k_i=k_{n-i+1}$, until it is needed (in Corollary \ref{tt*andreps} below) for
the relation with solutions $w$ of (\ref{ost}).   
}
of Examples 2.1, 3.6, 3.11 of \cite{GH2}, which
generally follow those of Kostant \cite{Ko59}.

We use the bilinear form $B(X,Y)=\tr XY$ to identify $\h^\ast$ with $\h$. Then
the basic weights are identified with $\eps_1,\dots,\eps_n\in\h$, where $\al_i(\eps_j)=\de_{ij}$. 
Explicitly:
\[
\eps_i= 
\diag
\left(
\left(1-\tfrac{i}{n+1}\right)
(
\underbrace{1,\dots,1}_i,
\underbrace{0,\dots,0}_{n+1-i})-
\tfrac{i}{n+1}(
\underbrace{0,\dots,0}_i,
\underbrace{1,\dots,1}_{n+1-i})
\right).
\]
We have $\rho=\eps_1+\cdots+\eps_n$. 
The weight lattice is
\[
\textstyle
P=\{ \sum_{i=1}^n v_i \eps_i \st \text{all $v_i\in\Z$} \}
\]
(see Remark \ref{realver} below).  The (fundamental)
Weyl chamber is
$C=\{ \sum_{i=1}^n v_i \eps_i  \st \text{all $v_i\ge 0$} \}$,
and the dominant weights are $P_+=P\cap C$. 
The (fundamental) Weyl alcove is
\[
\textstyle
A= \{ \sum_{i=1}^n v_i  \eps_i \st \text{all $v_i\ge 0$ and $\sum_{i=1}^n v_i\le 1$} \}.
\]

\begin{remark}\label{realver} 
In section 6 of \cite{GH2} we put $\h_\sharp=\{ h\in\h \st \text{all $\al_i(h)\in \R$} \}$,
so that $\ii \h_\sharp$ is the standard Cartan subalgebra of the compact
real form $\su_{n+1}$.  Then the integer lattice is
$I=\ii \h_\sharp \cap 2\pi\ii \Z^{n+1}$.  Thus the real roots
$\alpha^{\text{real}}=(2\pi\ii)^{-1}\alpha$ take integer values on the integer lattice, as do all
real weights.  The basic real weights are
\[
\La_i^{\text{real}}=
\tfrac{1}{2\pi\ii}
\left(1-\tfrac{i}{n+1}\right)
(x_0+\cdots+x_{i-1})-
\tfrac{1}{2\pi\ii}
\tfrac{i}{n+1}(x_{i}+\cdots+x_{n})
\]
$(1\le i\le n)$,
and the weight lattice is $W=\oplus_{i=1}^n \Z \La_i^{\text{real}}$.  With
these definitions the Weyl alcove is
\[
\fA= \{ y\in \ii \h_\sharp \st 
\text{all $\alpha^{\text{real}}_i\ge 0$ and $\psi^{\text{real}}(y)\le 1$} \}.
\]
To simplify the presentation in this article we use the convention \ll without $2\pi\ii$\rr (which
amounts to declaring that the exponential map is $X\mapsto e^{2\pi\ii\,X}$).  We
 obtain $P$ (instead of $W$)
and $A$ (instead of $\fA$). Our integer lattice is the set of integer diagonal
matrices in $\sl_{n+1}\C$.
\qed
\end{remark}

We recall (see \cite{Ka83}, \cite{PrSe86}) that the affine Kac-Moody algebra $\widehat{\sl}_{n+1}\C$ is an 
extension of the loop algebra $\La\sl_{n+1}\C$ by two additive generators, and that
the irreducible positive energy representations of $\widehat{\sl}_{n+1}\C$ of level $k$
are parametrized by dominant weights $(\La,k)$, where $\La$ is a dominant weight
of $\sl_{n+1}\C$ of level $k$. 

For nontrivial representations we have $k\in\N$.  The 
dominant weights
of $\sl_{n+1}\C$ of level $k$ are
\[
\textstyle
P_k=\{ \sum_{i=1}^n v_i \eps_i \in P_+ \st 
\sum_{i=1}^n v_i \le k
\}.
\]
The following fact is well known, but we give the short proof.

\begin{lemma}\label{pk}  We have
$
P_k + \rho = P_+ \cap (k+n+1) \mathring A
$
where $\mathring A$ denotes the interior of the Weyl alcove $A$.
\end{lemma}

\begin{proof}  Let $v=\sum_{i=1}^n v_i \eps_i$ with all $v_i\in\Z_{\ge 0}$.  Then:
(i) $v\in P_k+\rho$ iff $v_i\ge 1$ and $\sum_{i=1}^n (v_i-1) \le k$,
i.e.\ $v_i>0$ and $\sum_{i=1}^n v_i < k+n+1$;  
(ii) $v\in \cap (k+n+1) A$ iff 
 $v_i\ge 0$ and $\sum_{i=1}^n v_i \le k+n+1$.  
 Hence $v\in \cap (k+n+1) \mathring A$ iff 
 $v_i>  0$ and $\sum_{i=1}^n v_i < k+n+1$.  
Thus (i) and (ii) are equivalent.
\end{proof}

In view of this, we introduce the following notation:

\begin{definition}\label{theta}
Let 
$\mathring A_k=  \left( \tfrac{1}{k+n+1} P_+ \right) \cap\mathring A$.
Let $\th:\mathring A_k \to P_k + \rho$ be the identification
given by $\th(v)=(k+n+1)v\in P_+ \cap (k+n+1) \mathring A =
P_k + \rho$.
\end{definition}

Recall from section \ref{sto} that the Stokes data
is represented by a certain matrix $\Mwd=e^{\frac{2\pi\i}{n+1}(m+\rho)}$.  With the conventions of
\cite{GH2}, the corresponding Lie algebra element $\frac{2\pi\i}{n+1}(m+\rho)$
is in $\fA$ (see Remark \ref{realver}); with our current conventions we have
$\frac{1}{n+1}(m+\rho)$ in $A$. 

We now ask:

(i) for which $m$ does $\frac1{n+1}(m+\rho)$ lie
in the subset $\mathring A_k$, for some $k$ ?

(ii) in that case, what is the corresponding element of $P_k + \rho$?

\no The answers are:

\begin{theorem}\label{stokesandreps}   
Assume that $m$ arises from $k_0,\dots,k_n\in\Z_{\ge 0}$ through formula
(\ref{mandk}) (with $m_0+\cdots+m_n=0$). Then:

\no(a) $\frac{1}{n+1}(m+\rho) \in\mathring A_k$ for $k=\sum_{i=0}^n k_i$, and
 
\no (b) the corresponding element of $P_k + \rho$ is $(\sum_{i=1}^n k_i\eps_i) + \rho$.
\end{theorem}

\begin{proof}  First we observe that formula (\ref{mandk}) is equivalent to 
\begin{equation}\label{Mandk}
\textstyle
N \frac1{n+1}(m+\rho) = \rho + \sum_{i=1}^n k_i\eps_i.
\end{equation}
(To verify this, it suffices to apply each simple root  $\al_i$ to both sides,
then use $\al_i(m)=m_{i-1}-m_i$ and $\al_i(\rho)=1$.)
Next we put $k=\sum_{i=0}^nk_i$ (hence $N=n+1+k$).
Then (\ref{Mandk}) says that
$\th(\frac1{n+1}(m+\rho)) = \rho + \sum_{i=1}^n k_i\eps_i$,
where $\th$ is as in Definition \ref{theta}.
This gives both (a) and (b).
\end{proof}

Restricting now to the tt*-Toda situation (i.e.\ assuming $k_i=k_{n-i+1}$), we obtain:

\begin{corollary}\label{tt*andreps}  Assume that  $k_i\in\Z_{\ge 0}$
and $k_i=k_{n-i+1}$. Let $N=n+1+\sum_{i=0}^n k_i$.
Then there is a one-to-one correspondence between

\no
(i)  solutions $w$ of the tt*-Toda equation given by $\eta$ (as in Theorem \ref{GIL})

\no
(ii) Stokes data $\Mzw$ given by $m$  (as in Theorem \ref{GILstokes})

\no
(iii) positive energy representations of $\widehat{\sl}_{n+1}\C$ with dominant weights
$(\La,k)=(\sum_{i=1}^n k_i\eps_i,N-(n+1))$.
\end{corollary}

\begin{example}\label{sing} Let $k_0=1$ and $k_1=\cdots=k_n=0$. Then $N=n+2$ and $m=-\tfrac 1{n+2} \rho$. The corresponding representation has
dominant weight $(0,1)$; this is the basic representation of $\widehat{\sl}_{n+1}\C$.
Let us compute the Stokes parameters $s_1,\dots,s_n$, using Theorem \ref{GILstokes}.
These are
the elementary symmetric functions of 
$
e^{\frac{\pi\i}{n+2} n},
e^{\frac{\pi\i}{n+2} (n-2)},
\dots,
e^{\frac{\pi\i}{n+2} (-n)}.
$
When $n+1$ is even, they are the $(n+2)$-th roots of $-1$, excluding $-1$ itself. They are the roots of the polynomial
\[
\tfrac{x^{n+2} + 1}{x+1} = x^{n+1} - x^n + x^{n-1} - \cdots - x + 1.
\]
When $n+1$ is odd, they are the $(n+2)$-th roots of $1$, excluding $-1$, i.e.\  
the roots of the polynomial
\[
\tfrac{x^{n+2} - 1}{x+1} = x^{n+1} - x^n + x^{n-1} - \cdots + x - 1.
\]
In both cases, all $s_i=1$. 
\qed
\end{example}

To comment on the significance of Theorem \ref{stokesandreps}, one might say that it is hardly
surprising that a positive energy representation of $\widehat{\sl}_{n+1}\C$ can
be concocted artificially from the element $\frac1{n+1}(m+\rho)$ of the Weyl alcove,
as the Weyl alcove plays such a fundamental role in the theory of affine Lie algebras.
However, Corollary \ref{tt*andreps}  says that the (dominant weight of the) representation
is given {\em precisely} by the integers $k_i$ from which the solution $w$ was constructed.
Moreover the  positive energy representations give {\em all} the global solutions of (\ref{ost})
which are  generic (i.e.\ $m_{i-1}-m_i+1>0$)
and rational (i.e.\ $m_i\in\Q$).  These form an open dense subset of all
global solutions. Thus the
representations are tightly related to the solutions of the tt*-Toda equations
through our construction.

\begin{remark}  Positive energy representations of $\widehat{\sl}_{n+1}\C$
give (projective) representations of the loop group $\La\SL_{n+1}\C$.  We have seen (and it is well known) that $\La\SL_{n+1}\C$ plays an important role in solving the Toda equations. 
Thus one can expect a more
direct role for the representation associated to $w$ in Corollary \ref{tt*andreps}. 
Indeed, the
solutions are obtained by taking the Iwasawa factorization of
the holomorphic $\La\SL_{n+1}\C$-valued function $L$, and
this is equivalent to the Birkhoff factorization of
$c(L)^{-1}L$, where $c$ is the real form involution of $\La\SL_{n+1}\C$.
This should give
a determinant formula for the $\tau$-function of $w$, as
the Birkhoff factorization can be expressed in terms of infinite determinants given by $\tau$-functions. 
Although we have not pursued this, the existence
of such formulae is well known --- see section 6 of \cite{GIL2} for a brief explanation
of a  determinant formula due to Tracy and Widom \cite{TrWi98}. 
\qed
\end{remark}

\section{Relations with conformal field theory}\label{cft}

In this section we describe three ways in which the construction of
section \ref{alc} is relevant to conformal field theory. It is written mainly
for mathematicians who might not be familiar with physics,
but we hope that the ideas sketched here are not too inaccurate and might
also be of passing interest to physicists.

\no{\em 1. Topological-antitopological fusion}

According to Cecotti and Vafa (\cite{CeVa91},\cite{CeVa93}) it is the
solutions with \ll integer Stokes data\rr which represent physically realistic models.  
With our notation, this means solutions with integer 
Stokes parameters $s_i$.  Then the $s_i$ can be interpreted as counting Bogomolnyi solitons.

Furthermore, the $s_i$ appear as leading term coefficients in the asymptotics 
as $t\to\infty$ of the
corresponding solution $w$ (see \cite{Gu21} for a precise statement).
It follows from this (and Theorem \ref{GILstokes} above) that the global
solutions are characterized equally well by their asymptotics at $t=\infty$
as by their asymptotics at $t=0$.  This is another property predicted by 
Cecotti and Vafa, on the grounds that $w$ represents the
renormalization group flow between the chiral data at $t=0$ (in our notation,
the $k_i$ or $m_i$) and the soliton data at $t=\infty$ (the $s_i$). 

We have seen a solution of this type already in Example \ref{sing}, where
we have $k_0=1, k_1=\cdots=k_n=0$ and $m=-\tfrac 1{n+2} \rho$. 
All $s_i$ are equal to $1$ here.  
From the tt* point of view, this particular solution 
corresponds to the supersymmetric $A_{n+1}$ minimal model. Geometrically,
it corresponds to an unfolding $\frac1{n+2} x^{n+2} - tx$ of the $A_{n+1}$ singularity
$\frac1{n+2} x^{n+2}$.   

Other solutions of the tt* equations
with geometric interpretations are those corresponding to
the quantum cohomology of K\"ahler manifolds (or orbifolds).  
It is
implicit in \cite{CeVa91} that all such solutions are
expected to be  globally defined.
In the
(very special) case of the tt*-Toda equations, the basic example
is the  quantum cohomology of $\C P^n$, complex projective space. Here we have
$k_0=0, k_1=\cdots=k_n=-1$ and $m=-\rho$, and the
solution is indeed globally defined.  However, the assumption $k_i>-1$ is not satisfied
here, and Theorem \ref{qqpi} does not apply (in fact, $\Mzw$
is not diagonalizable).  Nevertheless the formula for
the Stokes numbers
in Theorem \ref{GILstokes} does apply, and it gives
$s_i=\binom{n+1}{i}$. 

Further examples (such as weighted projective spaces and their
hypersurfaces) can be found in \cite{GuLi12}.  In all cases 
the assumption $k_i>-1$ is violated.
 
With the prominent exception of the 
$A_{n+1}$ minimal model, however, the \ll integer Stokes data\rr solutions are 
generally not
of the type considered in section \ref{alc}.

\no{\em 2. The fusion ring}

The fusion ring of the WZW model $SU(n\!+\!1)_k$
is a certain  ring structure on the set of positive energy representations of 
$\widehat{\sl}_{n+1}\C$ of level $k$.  We refer to \cite{Ge91} for the
background, and \cite{Me18} for a treatment close to the
context of the current article.

The ring can be described succinctly (using the notation
of section \ref{alc}) as follows.  For a
positive energy representation with dominant weight $(\La,k)$, 
where $k\in\N$ and $\La\in P_k$,  a \ll special
element\rr is defined by
\[
t_\La = e^{2\pi\i \zeta_\La},
\quad
 \zeta_\La=
\tfrac{\La+\rho}{k+n+1}.
\]
Then the level $k$ fusion ideal $I_k(\SU_{n+1})$
of the
representation ring $R(\SU_{n+1})$  is defined by
\[
I_k(\SU_{n+1}) = \{ 
\text{representations whose characters vanish at all $t_\La, \La\in P_k$}
\}.
\]
The level $k$ fusion ring is then $R(\SU_{n+1})/I_k(\SU_{n+1})$.

Our observation concerning this is that 
\[
\zeta_\La = \tfrac1{n+1}(m+\rho),
\]
where $m$ corresponds to $\La=\sum_{i=1}^n k_i\eps_i$
as in section \ref{alc}.  This follows immediately from (\ref{Mandk}).
In other words, the special element 
$t_\La$ is precisely our matrix 
$\Mwd$
which represents the Stokes data of the holomorphic $1$-form $\eta(z)dz$
(see Remark \ref{antisym}).

We do not know a satisfactory explanation of this coincidence. On the one hand,
it is well known that fusion arises geometrically from \ll fusing\rr moduli spaces 
of flat $\SU_{n+1}$-connections over Riemann surfaces with a
common boundary component, and it is known that such moduli spaces
can be described in terms of monodromy data. 
On the other hand
the connections in sections \ref{intro}-\ref{sto} are {\em not} $\SU_{n+1}$-connections.

\no{\em 3. Minimal models}

It is well known (e.g.\ section 9.4 of \cite{PrSe86}) 
that the Virasoro algebra acts by intertwining operators on any $\widehat{\sl}_{n+1}\C$-module
of positive energy.   In this way a positive energy representation gives a representation
of the Virasoro algebra.  
Irreducible representations are classified according to their
central charge $c$ and conformal dimension $h$. 

Representations of the Virasoro algebra can be used to construct special examples of conformal field theories called minimal models; in a minimal model, the Hilbert space of the theory is a sum of finitely many irreducible representations, and the representations which occur are highly restricted.  The theory of these 
\ll Virasoro minimal models\rr is described in \cite{DMS97}.

More generally, representations of the $W$-algebra $W_{n+1}$ (see \cite{FKW92},\cite{BoSc93})
can be used to construct  
\ll $W_{n+1}$ minimal models\rrr.  (The  $W$-algebra $W_{2}$ is the Virasoro algebra.)
The theory of $W_{n+1}$ minimal models is described in \cite{BoSc93}, from which we quote the following result:

\begin{theorem}\label{bs} (\cite{BoSc93})
Let $p,\pp\in\N$ be coprime.  Let $\La^{(+)},\La^{(-)}$ be dominant weights of ${\sl}_{n+1}\C$.
Then there exists an irreducible representation of $W_{n+1}$ whose
central charge is
\begin{equation}\label{cformula}
c=n-n(n+1)(n+2) \tfrac{(\pp-p)^2}{p\pp}
\end{equation}
and whose conformal dimension $h$ is given by
\begin{equation}\label{hformula}
c-24h=n  -  12  \,  
\left\vert
\al_+ (\La^{(+)} + \rho) + \al_- (\La^{(-)} + \rho)
\right\vert^2,
\end{equation}
where $\rho$ is as in Theorem \ref{qqpi}, and
$\al_+ = \sqrt{\pp/p}, \al_- = -\sqrt{p/\pp}$.
\end{theorem}

Some comments on the notation of \cite{BoSc93} are in order before we
proceed further.
The central charge formula is (6.13) in \cite{BoSc93},  and the conformal dimension
formula is (6.74). The scalars $\al_+,\al_-$ are given just after (6.13) and in (6.75);
we have taken the positive square root for $\al_+$.

For the $W_{n+1}$ minimal model of type $(p,\pp)$ the dominant weights
$\La^{(+)},\La^{(-)}$
are restricted as follows:
\[
\La^{(+)}\in P_{p-(n+1)},
\quad
\La^{(-)}\in P_{\pp-(n+1)}.
\]
This is (6.76) in \cite{BoSc93}.  However, dominant weights which lie in the same orbit of
the centre of $SU_{n+1}$ should
be identified.  This is (6.77) in \cite{BoSc93}. We shall make the action explicit in a moment.
By definition, the minimal model of type  $(p,\pp)$ consists of the set of equivalence classes with respect to this
action. 

The special case $p=n+1,\pp=N=n+1+k=
n+1+\sum_{i=0}^n k_i$ was considered by Fredrickson and Neitzke in \cite{FrNeXX}, in
connection with Argyres-Douglas theories of type $(A_{n},A_{k-1}$).
(Our $n+1$ is called $K$ by them, and our $k=\sum_{i=0}^n k_i$ is called $N$ by them.)  
Here we have
\[
\La^{(+)}=0,
\quad
\La^{(-)}\in P_k.
\]
From now on we shall write $\La=\La^{(-)}$, in keeping with our earlier notation for
dominant weights (this should not be confused with the notation $\La= 
\al_+ \La^{(+)}+ \al_- \La^{(-)}$ in (6.73) of \cite{BoSc93}, which we shall
not use).

Formulae (\ref{cformula}) and (\ref{hformula}) become, in this case:
\begin{equation}\label{cformulaFN}
c =
n - \tfrac1N n (n+2) (N - (n+1))^2
\end{equation}
\begin{equation}\label{hformulaFN}
c-24h=n  -  12 \tfrac {n+1}N \,  \vert\La - \tfrac{N-(n+1)}{n+1} \rho \vert^2
\end{equation}

\begin{example}
The values $\La^{(+)}=\La^{(-)}=0$ are always permitted
in the $W_{n+1}$ minimal model, and for these weights we have $h=0$.  
However a feature of the special case $p=n+1$ is that $h\le 0$
for all weights in the model.  This follows from formula (\ref{hformulaMGTO}) below. As a consequence,
in this special case, the model cannot be unitary.
\qed
\end{example}

The weights $\La$ in the $(n+1,n+1+k)$ minimal model
are restricted to lie in $P_k$, but (as mentioned above) 
there is a further reduction given by dividing by the action of
the centre of $SU_{n+1}$, i.e.\ by the cyclic group
$\{I,\iota I,\dots,\iota^n I \}$ where $\iota$ is a primitive $(n+1)$-th root of unity.

\begin{proposition}
The action of $\iota$ on the weight $\La=\sum_{i=1}^n k_i\eps_i$
is given by 
\[
\iota\cdot\La= \sum_{i=1}^n k_{\si\cdot i} \, \eps_i
\]
where $\si$ is the cyclic permutation $\si=(01\cdots n)$. Here we are using the
notation $k_0,\dots,k_n$ as in section \ref{alc}, i.e.\ $N=n+1+\sum_{i=0}^n k_i$
and the subscript $i$ of $k_i$ is interpreted mod $n+1$. 
\end{proposition}

\begin{proof}  The essential ingredient here is the action of the centre on
the alcove $A$, which is explained in detail in \cite{TL99} 
and in section 2.4 of \cite{Me18}.  An
explicit formula for the case $SU_{n+1}$ can be found in 
section 4.4 of \cite{TL99}.  Applying this formula to our alcove
element $\frac1{n+1}(m+\rho)$, we see that the action corresponds to
cyclic permutation of the subscripts of the $m_i$ (without changing $\rho$).   From our formula
(\ref{mandk}), this corresponds to
cyclic permutation of the subscripts of the $k_i$.
\end{proof}

Thus the weights $\La=\sum_{i=1}^n v_i\eps_i$ which occur 
in the $W_{n+1}$ minimal model of type $(p,\pp)=(n+1,N-(n+1))$
are indexed by 
\ll cyclic $(n+1)$-partitions of $N-(n+1)$\rrr.  These may be enumerated as follows.

\begin{proposition}
Assume that $n+1$ and $N$ are coprime. Then there are 
$\frac 1N \binom N{n+1}$ 
equivalence classes of 
$(n+1)$-tuples $(k_0,\dots,k_n) \in \Z_{\ge0}^{n+1}$
such that $\sum_{i=0}^n k_i=N-(n+1)$, where 
the equivalence relation is defined by
$(k_0,\dots,k_n)\sim (l_0,\dots,l_n)$
if $l_i=k_{i-s}$ for some $s\in\N$ (all indices are mod $n+1$).
\end{proposition}

\begin{proof}  A reference from the combinatorics literature is Corollary 1 of \cite{RoKn10}, but we shall give a proof based on the properties of the holomorphic connection $\na=d+\eta(z)dz$, which
seems appropriate for the present context. 

The equation for parallel sections of $\na=d-\eta(z)^tdz$ can be written
\[
D
\bp
y_0 \\ \vdots \\ y_n
\ep
=
\eta(z)^t
\bp
y_0 \\ \vdots \\ y_n
\ep,
\quad
D=\tfrac d{dz}.
\]
Each $y_i$ satisfies a scalar o.d.e.\
\[
z^{-k_i} D z^{-k_{i-1}} D \cdots D z^{-k_{i-n}} D \,y_i = y_i
\]
where we interpret $k_{i-(n+1)}$ as $k_i$.  As $i$ varies, these scalar equations differ by cyclic permutations, and they are all equivalent
in an obvious sense --- more formally, they define isomorphic $D$-modules.  

To calculate the number of equivalence classes, we have to calculate
the number of (ordered) strings of symbols $D$ ($n+1$ times) and
$z^{-1}$ ($N-(n+1)=\sum_{i=0}^n k_i$ times) up to
cyclic equivalence.

Now, there are $\binom N{n+1}$ ways to choose the positions of the $D$'s.
This gives $\frac 1N \binom N{n+1}$ cyclic equivalence classes, 
{\em assuming that} the orbit of every string has $N$ distinct elements.  An
orbit has less than $N$ distinct elements if and only if it contains a sub-string of length $l$, containing $m$ $D$'s say, repeated $r$ times (with $l,r>1$). In that
case we have $lr=N$ and $mr=n+1$, but this is impossible
as $n+1$ and $N$ are coprime.
\end{proof}

Now we shall connect this to the construction of section \ref{alc}. 
Formula (\ref{Mandk}) can be written
\[
N \tfrac1{n+1}m = \La- \tfrac{N-(n+1)}{n+1}\rho
\]
where $\La=\sum_{i=1}^n k_i\eps_i$.  Using this,
(\ref{hformulaFN}) becomes
\begin{equation}\label{hformulaMGTO}
c-24h=n  -  12 \tfrac N{n+1} \,  \vert m \vert^2.
\end{equation}
Substituting for $c$ from (\ref{cformulaFN}), we can write the formula for $h$ as 
\begin{equation}
24h =-\tfrac1N n(n+2)  (N - (n+1))^2 + 12 \tfrac N{n+1} \vert m\vert^2.
\end{equation}
Alternatively, using the fact that
$
\vert\rho\vert^2=\tfrac 1{12} n(n+1)(n+2),
$
we can write
\begin{equation}
h= \tfrac {n+1}{2N} 
\left(
\left\vert \La- \tfrac{N-(n+1)}{n+1} \rho \right\vert^2 - 
\left\vert\tfrac{N-(n+1)}{n+1}  \rho \right\vert^2
\right).
\end{equation}

Fredrickson and Neitzke arrive at these considerations from a rather different starting point,
 namely a certain moduli space of Higgs bundles on $\C$.   
 The moduli space admits an action of $\C^\ast$, and the fixed points of this action are the
 forms (or Higgs fields) $\eta(z)dz$ with all $k_i\in\Z_{\ge 0}$.   In their notation the quantity
 $\frac N{n+1} \vert m\vert^2$ is called $\mu$, and it is identified with
 a \ll regulated norm\rr of  $\eta(z)dz$.   Our formula (\ref{hformulaMGTO}) is then
 \[
 c-24h=n-12\mu.
 \] 
 Thus we recover Theorem 5.3 of \cite{FrNeXX}.  This relation is
 the basis for the \ll somewhat mysterious\rr bijection (\cite{FrNeXX}, section 5)
 between Higgs fields
 and representations of $W_{n+1}$.  
 
Our construction in section \ref{alc} gives a (mathematical) explanation of this bijection.   Namely, it shows how the representation arises directly from  $\eta(z)dz$ by means of its Stokes data. Thus it is the Stokes data which provides the crucial link.

{\em
\noindent
Department of Mathematics, Faculty of Science and Engineering\newline
Waseda University\newline
3-4-1 Okubo, Shinjuku\newline
Tokyo 169-8555\newline
JAPAN
}

{\em
\noindent
College of Engineering\newline
Nihon University\newline
1 Nakagawara, Tamuramachi Tokusada, Koriyama\newline
Fukushima 963-8642\newline
JAPAN
}


\begin{thebibliography}{11}

\bibitem{BoSc93}
P. Bouwknegt and K. Schoutens,
\emph{W-symmetry in conformal field theory},
Physics Reports
\textbf{223}
(1993),
183--276.

\bibitem{CeVa91}
S. Cecotti and C. Vafa,
\emph{Topological---anti-topological fusion},
Nuclear Phys. B 
\textbf{367}
(1991),
359--461.

\bibitem{CeVa93}
S. Cecotti and C. Vafa,
\emph{On classification of $N=2$ supersymmetric theories},
Comm. Math. Phys.
\textbf{158}
(1993),
569--644.

\bibitem{DoPeWu98}
J.~Dorfmeister, F.~Pedit, and H.~Wu,
\emph{Weierstrass type representations of harmonic maps into symmetric spaces},
Comm. Anal. Geom.
\textbf{6}
(1998),
633--668.

\bibitem{Du93}
B. Dubrovin,
\emph{Geometry and integrability of 
topological-antitopological fusion},
Comm. Math. Phys.
\textbf{152}
(1993),
539--564.

\bibitem{FIKN06}
A.~S.~Fokas, A.~R.~Its, A.~A.~Kapaev, and V.~Y.~Novokshenov,
\emph{Painlev\'e Transcendents: The Riemann-Hilbert Approach},
Math. Surveys and Monographs 128,
Amer. Math. Soc.,
2006.

\bibitem{DMS97} 
P. Di Francesco, P. Mathieu, and D. S\'en\'echal
\emph{Conformal Field Theory},
Springer, 1997.

\bibitem{FrNeXX}
L. Fredrickson and A. Neitzke,
\emph{From $S^1$-fixed points to W-algebra representations},
arXiv:1709.06142 

\bibitem{FKW92}
E. Frenkel, V. Kac, and M. Wakimoto,
\emph{
Characters and fusion rules for W-algebras via quantized Drinfeld-Sokolov reduction},
Comm. Math. Phys. 
\textbf{147}
(1992),
295--328.

\bibitem{Ge91}
D. Gepner,
\emph{Fusion rings and geometry},
Comm. Math. Phys. 
\textbf{141}
(1991),
381--411.

\bibitem{Gu21} 
M. A. Guest,
\emph{Topological-antitopological fusion and the quantum cohomology of Grassmannians}, 
Jpn. J. Math. 
\textbf{16}
(2021), 
155--183.

\bibitem
{GH1}
M. A. Guest and N.-K. Ho, 
\emph{A Lie-theoretic description of the solution space of the tt*-Toda equations}, 
Math. Phys. Anal. Geom.
\textbf{20}
(2017),
article 24.

\bibitem
{GH2}
M. A. Guest and N.-K. Ho, 
\emph{Kostant, Steinberg, and the Stokes matrices of the tt*-Toda equations}, 
Sel. Math. New Ser.
\textbf{25}
(2019),
article 50.

\bibitem{GIL1}
M. A. Guest, A. R. Its, and C.-S. Lin,
\emph{Isomonodromy aspects of the tt*
equations of Cecotti and Vafa I. Stokes data},
Int. Math. Res. Notices
\textbf{2015}
(2015),
11745--11784.

\bibitem{GIL2}
M. A. Guest, A. R. Its, and C.-S. Lin,
\emph{Isomonodromy aspects of the tt*
equations of Cecotti and Vafa II. Riemann-Hilbert problem},
Comm. Math. Phys. 
\textbf{336}
(2015),
337--380.

\bibitem
{GIL3} 
M. A. Guest, A. R. Its and C. S. Lin, 
\emph{Isomonodromy aspects of the tt*
equations of Cecotti and Vafa
III.  Iwasawa factorization and asymptotics},  
Comm. Math. Phys.
\textbf{374}
(2020),
923--973.

\bibitem{GuLi12}
M. A. Guest and C.-S. Lin,
\emph{Some tt* structures and their integral Stokes data},
Comm. Number Theory Phys. 
\textbf{6}
(2012),
785--803.

\bibitem{GuLi14}
M. A. Guest and C.-S. Lin,
\emph{Nonlinear PDE aspects of the tt*
equations of Cecotti and Vafa},
J. reine angew. Math.
\textbf{689}
(2014),
1--32.
 
\bibitem{Ka83} 
V. G. Kac,
\emph{Infinite dimensional Lie algebras},
Progress in Mathematics 44,
Birkh\"auser, 1983.

\bibitem{Ko59}
B. Kostant,
\emph{The principal three-dimensional subgroup and the Betti numbers of a complex simple Lie group},
Am. J. Math.
\textbf{81}
(1959),
973--1032.

\bibitem{Me18}
E. Meinrenken,
\emph{Verlinde formulas for nonsimply connected groups},
Lie Groups, Geometry, and Representation Theory,
eds. V. Kac and V. Popov,
Progress in Mathematics 326, 
Birkh\"auser, 2018.
pp.~381--417.

\bibitem{MoXX} 
T. Mochizuki, 
\emph{Harmonic bundles and Toda lattices with opposite sign}, 
arXiv:1301.1718

\bibitem{Mo14} 
T. Mochizuki, 
\emph{Harmonic bundles and Toda lattices with opposite sign II},  
Comm. Math. Phys. 
\textbf{328} (2014), 1159--1198.

\bibitem{PrSe86} 
A. Pressley and G. B. Segal,
\emph{Loop Groups},
Oxford Univ. Press, 1986.

\bibitem{RoKn10} 
N. Robbins and A. Knopfmacher
\emph{Some properties of cyclic compositions},  
Fibonacci Quarterly  
\textbf{48} (2010), 249--255.

\bibitem
{TL99} 
V. Toledano Laredo, 
\emph{Positive energy representations of the loop groups of non-simply connected groups},  
Comm. Math. Phys.
\textbf{207}
(1999),
307--339.

\bibitem{TrWi98}
C.~A.~Tracy and H.~Widom, 
\emph{Asymptotics of a class of solutions to the cylindrical Toda equations},
Comm. Math. Phys.
\textbf{190}
(1998),
697--721.




\end{thebibliography}
\end{document}